\documentclass[12pt]{amsart}
\numberwithin{equation}{section}

\usepackage{amssymb}
\def\ca{{\mathcal A}}

\def\ce{{\mathcal E}}

\def\ch{{\mathcal H}}
\def\ci{{\mathcal I}}

\def\cs{{\mathcal S}}

\def\ga{{\mathfrak A}} 
\def\gb{{\mathfrak B}}

\def\gz{{\mathfrak Z}}

\def\bc{{\mathbb C}}

\def\bm{{\mathbb M}}
\def\bn{{\mathbb N}}
\def\bp{{\mathbb P}}

\def\bz{{\mathbb Z}}

\def\a{\alpha}
\def\b{\beta}
\def\g{\gamma}  \def\G{\Gamma}
\def\d{\delta}  

\def\eps{\varepsilon}
\def\io{\iota}
\def\l{\lambda} 

\def\m{\mu}
\def\p{\pi}
\def\n{\nu}
\def\r{\rho}
 
\def\t{\tau}
\def\f{\varphi}  \def\F{\Phi}
  \def\Th{\Theta}
\def\om{\omega} \def\Om{\Omega}

\newtheorem{thm}{Theorem}[section]
\newtheorem{lem}[thm]{Lemma}

\newtheorem{prop}[thm]{Proposition}
\newtheorem{defin}[thm]{Definition}
\newtheorem{rem}[thm]{Remark}

\def\aut{\mathop{\rm Aut}}
\def\carf{\mathop{\rm CAR}}

\def\spec{\mathop{\rm spec}}

\newcommand{\ty}[1]{\mathop{\rm {#1}}}
\def\di{\mathop{\rm d}\!}
\def\conv{\mathop{\rm conv}}

\def\ad{\mathop{\rm ad}}
\def\slim{\mathop{\rm s-lim}}
\def\idd{{1}\!\!{\rm I}}

\newcommand{\nn}{\nonumber}

\begin{document}

\title[de finetti theorem]
{de finetti theorem on the CAR algebra}
\author{Vitonofrio Crismale}
\address{Vitonofrio Crismale\\
Dipartimento di Matematica\\
Universit\`{a} degli studi di Bari\\
Via E. Orabona, 4, 70125 Bari, Italy}
\email{\texttt{crismalev@dm.uniba.it}}
\author{Francesco Fidaleo}
\address{Francesco Fidaleo\\
Dipartimento di Matematica \\
Universit\`{a} di Roma Tor Vergata\\
Via della Ricerca Scientifica 1, Roma 00133, Italy} \email{{\tt
fidaleo@mat.uniroma2.it}}
\date{\today}

\begin{abstract}
\vskip 0.3cm The symmetric states on a quasi local $C^*$--algebra
on the infinite set of indices $J$ are those invariant under the
action of the group of the permutations moving only a finite, but
arbitrary, number of elements of $J$. The celebrated De Finetti
Theorem describes the structure of the symmetric states (i.e.
exchangeable probability measures) in classical probability. In
the present paper we extend De Finetti Theorem to the case of the
CAR algebra, that is for physical systems describing Fermions.
Namely, after showing that a symmetric state is automatically even
under the natural action of the parity automorphism, we prove that
the compact convex set of such states is a Choquet simplex, whose
extremal (i.e. ergodic w.r.t. the action of the group of
permutations previously described) are precisely the product
states in the sense of Araki--Moriya. In order to do that, we also
prove some ergodic properties naturally enjoyed by the symmetric
states which have a self--containing interest.

\end{abstract}

\maketitle

\section{introduction}
\label{sec1}

Exchangeable or, equivalently, symmetrically dependent sequences
of random variables and symmetric states have been investigated in
a wide way both in Probability Theory and Operator Algebras.

After De Finetti's pioneering work \cite{DeF} for 2--point valued
random variables, it has been shown that more and more general
sequences of exchangeable random variables are mixtures of
independent identical distributed (i.i.d. for short) sequences.
One of its most general version in classical probability was
obtained by Hewitt and Savage in \cite{HS} for exchangeable random
variables distributed on $X=E\times E\times\ldots$, $E$ being a
compact Hausdorff space.

A noncommutative extension of this result for infinite tensor
product $\ga$ of a single $C^{*}$-algebra $\gb$, was given by
St{\o}rmer in \cite{St2}, where it is shown that the symmetric
states on $\ga$ form a simplex whose $*$--weakly closed set of
extremal points is made exactly by the product states. Indeed
St{\o}rmer's theorem, when reduced to abelian $C^{*}$-algebras is
nothing else than Hewitt--Savage result.

Other similar characterizations can be found in \cite{FLV, HM} for
Boson quantum systems, or in \cite{AL, Fr} in the case of
continuous index set. More recently in \cite{K, KS, L}, the
authors got some De Finetti's type theorems in Free Probability,
whereas in \cite{ABCL} some results on the structure of symmetric
(exchangeable) states on general $C^{*}$-algebras have been
obtained.

Besides infinite sequences case, Diaconis and Freedman in
\cite{DF} obtained a De Finetti theorem  for finite sequences of
exchangeable random variables. Namely they showed that the first
$k$ random variables of a permutation invariant distribution of
$n$ random variables can be \emph{approximated} by a convex
combination of $k$ i.i.d. random variables, the error being of
order $O(k^{2}/n)$. This has been the starting point for a recent
intensive investigation of finite De Finetti results in Quantum
Information Theory and the problem of the Entanglement. Unfortunately, when one generalizes to the
tensor product case such a result, it comes out that the
approximating error, contrarily to the classical situation,
depends on the dimension of the state space (see, e.g. \cite{CKMR}
for details). Hence, even if a general extension to infinite
dimensional quantum systems does not exist, some precise estimates,
independent of the dimension, have been evaluated for many
concrete physical classes. In particular, the
dependence on the local dimension is removed, for example,
when there is a bound on the number of the ways in which the system is measured, or
the $n$--particle reduced density matrix is separable \cite{CT},
when one treats an exponential version of the theorem in the case
of coherent states \cite{KM}, or when one takes orthogonal
invariant states \cite{LC}. Moreover, an analogous result can be
found in the so called Quantum Key Distribution (QKD), when one
deals with Gaussian states against general attacks \cite{RC}, or
when one modifies the symmetric hypothesis in a fully compatible way
with continuous--variables QKD protocols \cite{LKGC}.

Although none of the above mentioned results concern Fermions,
nowadays there is a rapidly increasing investigation of properties
and models dealing with physical systems describing such
particles. We mention for example the following issues which is
far to be complete. The investigation of the ground states of
lattice systems based on anticommutation annihilators \cite{M}.
The introduction and the study of the notion of the product state,
and the application to general thermodynamical properties of Fermi
lattice systems (see e.g. \cite{AM1, AM2} and the references cited
therein). The connection with the Markov Property, Quantum
Statistical Systems, Quantum Information Theory and Entanglement,
of chains of Fermi systems \cite{AFM, F1}. Disordered systems
based on Fermions \cite{BF}.

Thus, for the natural applications to Quantum Physics and also the
general implication in Quantum Probability, it is then natural to
address the study of the structure of the symmetric states on the
Fermi algebra, that is the CAR $C^{*}$-algebra. Up to the knowledge
of the authors, no result concerning the systematic study of this
subject is present in literature. This is the object of the
present paper. In more details, our purpose consists in:
\begin{enumerate}
    \item characterizing the extremal points of the convex
    set of symmetric states on the CAR algebra,
    \item showing that every symmetric state is the barycenter of a
    unique maximal Radon measure which is pseudosupported on the extremal states,
    \item proving that the extremal states form a $*$--weakly closed
    subset (thus the symmetric states are isomorphic to the
    regular probability measures on a compact Hausdorff space
    \cite{B}),
     \item determining the type of the von Neumann factors generated by the extremal (i.e. product) states.
\end{enumerate}
One of the main tools in De Finetti theorem for infinite tensor
product of $C^{*}$--algebras \cite{St2} is the asymptotic
Abelianess property with respect to the permutation group. In the
CAR algebra this property is not satisfied, because of the
anticommutation relations between spatially separated operators.
As a consequence, the results relative to the structure of
symmetric states in \cite{St2} can not be directly imported in our
case. Hence, after verifying that the group of permutations which
fix all but a finite number of the points in an arbitrary set $J$
(denoted by $\bp_J$), acts as a group of automorphisms on
$\carf(J)$ (the CAR algebra on $J$), we establish a result which
plays a crucial role in the sequel. Each symmetric (i.e. invariant
under the action of the group of permutations above) state on
$\carf(J)$ is even. This property is exploited throughout the
paper in order to obtain the results listed above. Indeed, it is
used in Theorems 5.3 and 6.1 to prove the equivalence between an
extremal symmetric state on $\carf(J)$ and a product state in the
Araki--Moriya sense (the product being constructed starting by a
single even state on $\bm_{2}(\bc)$), thus reaching point (1).

Further, the even property makes sure that the couple ($\carf(J),
\bp_J$) is $\bp_J$--Abelian (see Theorem \ref{mainer1}). This
allows to prove that the $*$--weakly compact convex set of
symmetric states is a Choquet simplex, and, consequently, to
obtain the ergodic decomposition presented in (2). The unique
decomposing measure is supported on the extremal states when $J$
is countable, and pseudosupported when $J$ is uncountable (see
Theorem \ref{csyx} and the final discussion in Section
\ref{sec6}).

As a consequence of the above discussion, one has that each
symmetric state is (at least in the countable case) a mixture of
product states of Araki--Moriya. This statement, which is our main
result, can be seen as the extension of De Finetti's theorem to
the CAR algebra.

The ending parts of Sections 5 and 6 are aimed to reach points (3)
and (4). In particular, some results due to St{\o}rmer in
\cite{St2} and the identification of every even state on
$\bm_{2}(\bc)$ with a single point of a closed segment (Lemma
2.2), give the extremal (i.e. product) states are a $*$--weakly
closed subset. Furthermore, the very special structure of a
product state, allows to achieve point (4), as Propositions 5.7,
5.8 and the relative following discussions.

As stressed above, the Fermi algebra is not asymptotically Abelian
with respect to $\bp_J$, but the even nature of a symmetric state
yields a "weak" asymptotic Abelianess property, see (ii) of
Theorem \ref{mainer}. This result, coupled with the property that
$\bp_J$ acts as a large group of automorphisms (cf. Definition
\ref{lgau} and Theorem \ref{mainer1}), provides some ergodic
properties (cf. Proposition \ref{avlgu} and Proposition
\ref{scaz}) for averages and orbits of symmetric states, which,
even not used for establishing the main results presented in
paper, may have a self--containing interest.

\section{the CAR algebra}
\label{sec2}

Denote by $[a,b]:=ab-ba$,
$\{a,b\}:=ab+ba$, the commutator and anticommutator between elements
$a$, $b$, respectively.

We start by quickly reviewing the basic properties of the Fermion
$C^{*}$-algebra, which, due to Pauli Exclusion Principle, is
generated by the annihilation and creation operators satisfying
the Canonical Anticommutation Relations. Indeed, let $J$ be an
arbitrary set. The {\it Canonical Anticommutation Relations} (CAR
for short) algebra over  $J$ is the $C^{*}$--algebra $\carf(J)$
with the identity $\idd$ generated by the set $\{a_j,
a^{\dagger}_j\mid j\in J\}$ (i.e. the Fermi annihilators and
creators respectively), and the relations
\begin{equation*}
(a_{j})^{*}=a^{\dagger}_{j}\,,\,\,\{a^{\dagger}_{j},a_{k}\}=\d_{jk}\idd\,,\,\,
\{a_{j},a_{k}\}=\{a^{\dagger}_{j},a^{\dagger}_{k}\}=0\,,\,\,j,k\in J\,.
\end{equation*}
The parity automorphism $\Th$ acts on the generators as
$$
\Th(a_{j})=-a_{j}\,,\,\,\Th(a^{\dagger}_{j})=-a^{\dagger}_{j}\,,\quad
j\in J\,
$$
and induces on $\carf(J)$ a $\bz_{2}$--grading. This grading
yields $\carf(J)=\carf(J)_{+} \oplus\carf(J)_{-}$, where
\begin{align*}
&\carf(J)_{+}:=\{a\in\carf(J) \mid\, \Th(a)=a\}\,,\\
&\carf(J)_{-}:=\{a\in\carf(J) \mid\, \Th(a)=-a\}\,.
\end{align*}
Elements  in $\carf(J)_+$ and in $\carf(J)_-$ are called
{\it even} and {\it odd}, respectively.

Notice that, by definition,
\begin{equation*}
\carf(J)=\overline{\carf{}_0(J)}\,,
\end{equation*}
where
$$
\carf{}_0(J):=\bigcup\{\carf(I)\mid\, I\subset J\,\text{finite}\,\}
$$
is the (dense) subalgebra of the {\it localized elements}.

A map $T:\ca_1\to\ca_2$ between $C^*$--algebras with
$\bz_{2}$--gradings $\Th_1$ and $\Th_2$, is said to be {\it even}
if it is grading--equivariant:
$$
T\circ\Th_1=\Th_2\circ T\,.
$$
The previous definition, applied to states $\f\in\cs(\carf(J))$,
leads to $\f\circ\Th=\f$, that is $\f$ is even if and only if it is
$\Th$--invariant.

When the
index set $J$ is countable, the CAR algebra $\carf(J)$ is isomorphic to the
$C^{*}$--infinite
tensor product of $J$--copies of $\bm_{2}(\bc)$,
\begin{equation}
\label{jkw}
\carf(J)\sim\overline{\bigotimes_{J}\bm_{2}(\bc)}^{C^*}\,.
\end{equation}
Such an isomorphism is established by a Jordan--Klein--Wigner
transformation, as shown in \cite{T3}, Exercise XIV. We briefly report it for the
convenience of the reader. Fix any enumeration
$j=1,2,\dots$ of the set $J$. Let
$U_{j}:=a_{j}a_{j}^{\dagger}-a_{j}^{\dagger}a_{j}$,
$j=1,2,\dots$\,. Put $V_{0}:=\idd$, ${
V_{j}:=\prod_{n=1}^{j}U_{n}}$, and denote
\begin{align}
\label{kw}
&e_{11}(j):=a_{j}a_{j}^{\dagger}\,,\quad e_{12}(j):=V_{j-1}a_{j}\,,\nn\\
&e_{21}(j):=V_{j-1}a_{j}^{\dagger}\,,\quad e_{22}(j):=a_{j}^{\dagger}a_{j}\,.
\end{align}
$\{e_{kl}(j)\,|\,k,l=1,2\}_{j\in J}$ provides a system of
commuting matrix units in $\carf(\bn)$. In order to
obtain
$$
\carf(\bn)\sim\overline{\bigotimes_{\bn}\bm_2(\bc)}^{C^*}\,,
$$
fix any segment $[1,l]\subset\bn$ and consider the system of matrix units localized in $r\in\bn$
$$
\{\eps_{i_rj_r}(r)\}_{i_rj_r=1,2}\subset\overline{\bigotimes_{\bn}\bm_2(\bc)}^{C^*}\,,
$$
together with
the system of matrix units
$$
\{e_{i_rj_r}(r)\}_{i_rj_r=1,2}\subset\carf(\bn)
$$
arising from the Jordan--Klein--Wigner construction \eqref{kw}.
The above isomorphism is simply described by
\begin{equation}
\label{iscartens}
e_{i_1j_1}(1)\cdots e_{i_lj_l}(l)\mapsto
\eps_{i_1j_1}(1)\otimes\cdots\otimes\eps_{i_lj_l}(l)
\end{equation}
for each $i_k,j_k=1,2$, $k=1,2,\dots,l$ and $l\in\bn$.
\begin{rem}
\label{rrlocc} Notice that, fixed $r\in\bn$, $e_{i_rj_r}(r)$ is
localized in the segment $[1,r]$ and, moreover, $e_{i_ri_r}(r)$ is
always localized in the site $r$. But $e_{i_rj_r}(r)$ is not
necessarily localized in the site $r$ if $i_r\neq j_r$.
\end{rem}
Anyone of such isomorphisms depends on a predefined order of the
countable index set $J$. Thus, it cannot be directly used to
investigate the exchangeable properties of the states under
consideration.

Thanks to \eqref{jkw}, $\carf(J)$ has a unique tracial state $\t$
as the extension of the unique tracial state on $\carf(I)$,
$|I|<+\infty$. Let $I\subset J$ be a finite set and
$\f\in\cs(\carf(J))$. Then there exists a unique positive element
$T\in\carf(I)$ such that
$\f\lceil_{\carf(I)}=\t\lceil_{\carf(I)}(T\,{\bf\cdot}\,)$. The
element $T$ is called the {\it adjusted density matrix} of
$\f\lceil_{\carf(I)}$. For the standard applications to quantum
statistical mechanics, one also uses the density matrix w.r.t.
the unnormalized trace.

We recall here the description of product state (cf. \cite{AM2}).
We start with the case of finite sets. Namely, let $I_1,I_2\subset
J$ with $|I_1|,|I_2|<\infty$ and $I_1\bigcap I_2=\emptyset$. Fix
$\f_{1}\in\cs(\carf(I_1))$, $\f_{2}\in\cs(\carf(I_2))$. If at
least one among them is even, then according to Theorem 1 of
\cite{AM2}, the product state extension (called {\it product
state} for short) $\f\in\cs(\carf(I_1\bigcup I_2))$ is uniquely
defined. We write, with an abuse of notation, $\f=\f_1\f_2$. Let
$T_1\in\carf(I_1)$, $T_2\in\carf(I_2)$ be the adjusted densities
relative to $\f_{1}\in\cs(\carf(I_1))$,
$\f_{2}\in\cs(\carf(I_2))$, respectively. As at least one among
$T_{1}$ and $T_{2}$ is even, $[T_{1},T_{2}]=0$ and $T:=T_{1}T_{2}$
is a well defined positive element of $\carf(I_1\bigcup I_2)$
which is precisely the density matrix of $\f=\f_1\f_2$. The
product state $\f\in\cs(\carf(I_1\cup I_2))$ is even if and only
if $\f_1$ and $\f_2$ are both even.

Now we pass to the
description of the product state on $\carf(J)$ symbolically written as
$$
\f:=\prod_{j\in J}\r\,,
$$
where $\r$ is a single even state on $\bm_2(\bc)\sim\carf(\{j\})$.
For $j\in J$ denote $\io_j:\bm_2(\bc)\to\carf(J)$ the
corresponding embedding. For each finite subset
$I:=\{j_1,\dots,j_{|I|}\}\subset J$, let $\f_I\in\cs(\carf(I))$ be
the product state given, on the elementary generators, by
$$
\f_I(\io_{j_1}(A_1)\cdots\io_{j_{|I|}}(A_{|I|}))=\prod_{k=1}^{|I|}\r(A_{k})\,,
$$
where $A_1,\dots,A_{|I|}\in\bm_2(\bc)$. If $I_1\subset I_2$, it is
immediate to see that $\f_{I_2}\lceil_{\carf(I_1)}=\f_{I_1}$. So
the direct limit $\underrightarrow{\lim } \f_I$, when $I\uparrow
J$ is a well defined state on the dense $*$--algebra of the
localized elements $\carf_0(J)$, which extends by continuity to a
state $\f$ which is the product state of a single even state
$\r\in\bm_2(\bc)$. A necessarily even product state
${\displaystyle\f=\prod_J\r}$ is then uniquely determined by the
even state $\r$, and the next lemma shows that each even state on
$\bm_2(\bc)$ can be seen as a single point of a closed segment.
\begin{lem}
\label{matr1}
For every even state $\r$ on $\bm_2(\bc)$, there
exists a unique $\m\in[0,1]$ such that
\begin{equation*}
\rho
\left(
\begin{array}{ll}
a & b \\
c & d
\end{array}
\right) =\mu a+\left( 1-\mu \right) d
\end{equation*}
\end{lem}
\begin{proof}
By the usual identification of $\carf(\{j\})$, $j\in J$ with
$\bm_2(\bc)$, one sees that for every state $\r$ on $\bm_2(\bc)$,
there exists a unique positive matrix $T$ such that
$\r=\t(T\,{\bf\cdot}\,)$. In particular $T=\left(
\begin{array}{ll}
\mu  & b+\imath f \\
b-\imath f & 1-\mu
\end{array}
\right) $, with $\m(1-\m)-(b^2+f^2)\geq0$, hence $\m\in[0,1]$. If
$\r$ is even, then $0=\rho \left(
\begin{array}{ll}
0 & 1 \\
0 & 0
\end{array}
\right)= b-\imath f$, hence $b=f=0$. This ends the proof.
\end{proof}
We refer to the above $\m$ and $1-\m$ as the eigenvalues of
the even state $\r_\m$ on $\bm_2(\bc)$, the latter inherited by
the $\bz_2$--grading arising from $\bm_2(\bc)\sim\carf(\{j\})$.

Let $\r_\m$ be an even state as before, and denote $\f_\m$,
$\om_\m$ the corresponding product states on $\ga:=\carf(\bn)$ and
${\displaystyle\gb:=\overline{\bigotimes_{\bn}\bm_2(\bc)}^{C^*}}$,
respectively. Denote $\g:\ga\to\gb$ the isomorphism described via
\eqref{iscartens}.
\begin{lem}
\label{matr2}
Under the above notations, for each $\m\in[0,1]$ we get $\f_\m=\om_\m\circ\g$.
\end{lem}
\begin{proof}
By the definition of the product states on $\ga$ and $\gb$, it is
enough to check the result for the system of the matrix units. Let
$\{e_{kl}(j)|k,l=1,2\}_{j\in\bn}$,
$\{\eps_{kl}(j)|k,l=1,2\}_{j\in\bn}$ be the canonical systems of
matrix units of $\carf(\bn)$ and
${\overline{\bigotimes_{\bn}\bm_2(\bc)}^{C^*}}$ described above.
By taking into account that $\f_\m$ is even and Remark
\ref{rrlocc}, we get for the restrictions to any segment $[1,l]$,
\begin{align*}
\f_\m(e_{i_1j_1}(1)\cdots e_{i_lj_l}(l))&=\f_\m(e_{i_1i_1}(1)\cdots e_{i_li_l}(l))
\d_{i_1j_1}\cdots \d_{i_lj_l}\\
=&\r_\m(e_{i_1i_1}(1))\cdots\r_\m(e_{i_li_l}(l))\\
=&\om_\m(\eps_{i_1i_1}(1)\otimes\cdots\otimes\eps_{i_li_l}(l))
\d_{i_1j_1}\cdots \d_{i_lj_l}\\
=&\om_\m(\eps_{i_1j_1}(1)\otimes\cdots\otimes\eps_{i_lj_l}(l))\,.
\end{align*}
\end{proof}

\section{the group of the permutations and its action on the CAR algebra}

We firstly present a result, probably known to the experts, crucial
in the sequel. Let
$$
G=\bigcup_{\a\in A}G_\a\,,
$$
where $A$ is a directed set, and $G_\a$, $\a\in A$ are finite
subgroups of the group $G$ such that $\a<\b$ implies $G_\a\subset
G_\b$. Consider a unitary representation $\{U(g)\mid\, g\in G\}$ of
$G$ acting on a Hilbert space $\ch$. Denote $E_\a$, $E$ the
selfadjoint projections onto the subspaces of $\ch$ consisting of
the invariant vectors under the action of $G_\a$ and $G$,
respectively. Of course, the net $\{E_\a\mid \,\a\in A\}$ is
decreasing. It is straightforward to see (cf. \cite{SZ}, Section 2.17) that it converges in the
strong operator topology to a projection
$$
P:=\slim_\a E_\a\,.
$$
In general, $P\geq E$. In addition, it is a standard fact to verify that
$$
E_\a=\frac1{|G_\a|}\sum_{g\in G_\a}U(g)\,.
$$
Here, we give the analogue of the von Neumann Ergodic Theorem  for
the case under consideration.
\begin{prop}
\label{fn} Under the above notations, we get
$$
\slim_\a E_\a=E\,.
$$
\end{prop}
\begin{proof}
We have only to prove that $P\leq E$. Fix $h\in G$ and $\xi\in\ch$
such that
$$
\xi=\lim_\a\frac1{|G_\a|}\sum_{g\in G_\a}U(g)\xi\,.
$$
As $A$ is a directed set and the sequence of the groups
$\{G_\a\mid\, \a\in A\}$ is increasing,  there exists $\a_h \in A$
such that $\a>\a_h$ implies $h\in G_\a$. Then, after a standard
change of variables in the sum, we obtain
\begin{align*}
U(h)\xi=&U(h)\lim_\a\frac1{|G_\a|}\sum_{g\in G_\a}U(g)\xi
=U(h)\lim_{\a>\a_h}\frac1{|G_\a|}\sum_{g\in G_\a}U(g)\xi\\
=&\lim_{\a>\a_h}\frac1{|G_\a|}\sum_{g\in G_\a}U(hg)\xi
=\lim_{\a>\a_h}\frac1{|G_\a|}\sum_{g\in G_\a}U(g)\xi\\
=&\lim_\a\frac1{|G_\a|}\sum_{g\in G_\a}U(g)\xi=\xi\,.
\end{align*}
\end{proof}

We now introduce and recall some notations and definitions which
will be used throughout the paper.
Let $G$ be a group and $\ga$ a $C^*$--algebra which we suppose
always to be unital. One says that $G$ acts as a group of
automorphisms of $\ga$ if there is a representation $\a: g\in
G\mapsto\a_g\in\aut(\ga)$. The state $\f\in\cs(\ga)$ is called
$G$--invariant if $\f=\f\circ\a_g$ for each $g\in G$. The subset
$\cs_G(\ga)$ of the $G$--invariant states is $*$--weakly compact
in $\cs(\ga)$, and its extremal points are called {\it ergodic
states} (w.r.t.the action of $G$). For $G$ acting as a group of
automorphisms of $\ga$ and a state $\f\in\cs_G(\ga)$,
$(\pi_\f,\ch_\f,U_\f,\Omega_\f)$ is the GNS covariant quadruple
canonically associated to $\f$ (see, e.g. \cite{BR1, T1}). If
$(\pi_\f,\ch_\f,\Omega_\f)$ is the GNS triple associated to $\f$,
the unitary representation $U_\f$ of $G$ on $\ch_\f$ is uniquely
determined by
\begin{align*}
&\pi_\f(\a_g(A))=U_\f(g)\pi_\f(A)U_\f(g)^{-1}\,,\\
&U_\f(g)\Om_\f=\Om_\f\,,\quad A\in\ga\,, g\in G\,.
\end{align*}
If $\f\in\cs_G(\ga)$, by $\gb_G(\f):=(\gz_\f)^\a$ we denote the
fixed point algebra of the center
$$
\gz_\f:=\pi_\f(\ga)''\bigcap\pi_\f(\ga)'
$$
under the adjoint action $\ad(U_\f)$ of $G$. We will refer to the set
$$
\ch_\f^{G}:=\{\xi\in\ch_\f\mid\, U_\f(g)\xi=\xi\,,g\in G\}\
$$
and $E_\f$ as the (closed) subspace of $\ch_\f$ of the invariant vectors
w.r.t. the action of $G$, and the relative selfadjoint projection onto it, respectively.

Let $(\ga,G)$ be a $C^*$--dynamical system as above, together with
$\f\in\cs_G(\ga)$. The invariant state $\f$ is said to be
$G$--{\it abelian} if all the operators $E_\f\pi_\f(\ga)E_\f$
mutually commute. The $C^*$--dynamical system $(\ga,G)$ is
$G$--abelian if $\f$ is $G$--abelian for each $\f\in\cs_G(\ga)$.

Let $J$ be any set. By definition the group of the permutations
$\bp_J$ of $J$ is made by those permutations leaving fixed all the
elements of $J$ but a finite number of them. Then it is the direct
limit of the (sub)groups of the permutations $\bp_I$, $I$ running
on all the finite subsets of $J$, that is
$$
\bp_J:=\bigcup\{\bp_I\mid\, I\subset J\,\text{finite}\,\}\,.
$$

It is expected that the group of the permutations $\bp_J$ acts, in
a natural way, as a group of automorphisms of $\carf(J)$. However
this is the case, according to the following
\begin{prop}
The map $g\in\bp_J\mapsto a_{g^{-1}j}\in\carf(J)$, $j\in J$, extends to an action
$g\in\bp_J\mapsto\a_g\in\aut(\carf(J))$ by automorphisms of
$\carf(J)$.
\end{prop}
\begin{proof}
By a standard argument, we have $\carf(J)\equiv\ga(\ell^2(J))$,
that is the CAR algebra on $\ell^2(J)$ under the notations in Section
5.2.2 of \cite{BR2}. Our action on the indices is nothing but a
unitary action on $\ell^2(J)$. This means that $\bp_J$ acts as a
group of Bogoliubov automorphisms of $\carf(J)$.
\end{proof}
We now report a result crucial in the sequel.\footnote{Compare with the connected estimation in Theorem 13 of \cite{DF}, concerning the classical case.} Denote
${\bf n}:=\{1,\dots,n\}$ the finite set made of exactly $n$
elements. If $m\leq n$, ${\bf m}$ can be considered, in a
canonical way, as a subset of ${\bf n}$.
\begin{lem}
\label{stiro}
Let $1\leq m,n<N$. Then, for some constant  $c(m,n)$ depending only on $m,n$, we have
\begin{equation*}
\frac{\left|\{g\in\bp_ {{\bf N}}\mid\,{\bf m}\cap g{\bf
n}\neq\emptyset\}\right|}{(N-1)!}\leq c(m,n)\,.
\end{equation*}
\end{lem}
\begin{proof}
Set $A:=\left|\{g\in\bp_ {{\bf N}}\mid\,{\bf m}\cap g{\bf
n}\neq\emptyset\}^c\right|$, and
\begin{align*}
\G=&(N-m)[\ln(N-m)-1]+(N-n)[\ln(N-n)-1]\\
-&(N-m-n)[\ln(N-m-n)-1]-N(\ln N-1)\,.
\end{align*}
It is straightforwardly seen that
$$
\frac{A}{N!}=\frac{\left( N-m\right) !(N-n)!}{(N-m-n)!N!}
\approx\sqrt{\frac{(N-m)(N-n)}{N(N-m-n)}}e^{\G}\,,
$$
after using the Stirling formula, for $N\to+\infty$. By retaining
only the leading terms up to the order $1/N$, we get
$$
\frac{A}{N!}\approx e^{-\frac{mn}N}\approx 1-\frac{mn}N
$$
which leads to
$$
\frac{\left|\{g\in\bp_ {{\bf N}}\mid\,{\bf m}\cap g{\bf
n}\neq\emptyset\}\right|}{N!} =1-\frac{A}{N!}\approx\frac{mn}N\,.
$$
\end{proof}

We end the section reporting the definition (cf. \cite{St1},
Definition 3.3) concerning the action of a group as a {\it Large
Group of Automorphisms}.
\begin{defin}
\label{lgau} Let $g\in G\mapsto\a_g\in\aut(\ga)$ be an action of a
group $G$ on the $C^*$--algebra $\ga$. We say that $G$ is
represented (or acts) as a large group of automorphisms if, for
each selfadjoint $A$ and each $\f\in\cs_{G}(\ga)$
$$
\overline{\conv\left(\{\pi_\f(\a_g(A))\mid \,g\in
G\}\right)}\bigcap\pi_\f(\ga)'\neq\emptyset\,.
$$
\end{defin}
In the next section we will establish that $\bp_J$ acts on
$\carf(J)$ as a large group of automorphisms (cf. Theorem
\ref{mainer1}). We underline that this property is not directly used
for the main result of the paper concerning the structure of the
symmetric (i.e. invariant under the action of $\bp_J$) states. It
comes out only for constructing a conditional expectation from the
GNS von Neumann algebra of $\carf(J)$ onto the invariant elements
of the center, that allows to obtain some convergence results
which may have some interest in general (see Proposition
\ref{avlgu} and Proposition \ref{scaz}).

\section{symmetric states on the CAR algebra}

A state $\f\in\cs(\carf(J))$ is called {\it symmetric} if it is
invariant under the action of the group $\bp_J$ of all the finite
permutations of the set $J$. Following the notation introduced
above, $\cs_{\bp_J}(\carf(J))$ denotes the $*$--weakly compact
subset of all the symmetric states of $\carf(J)$. Furthermore we
refer to $\ce\left(\cs_{\bp_J}(\carf(J))\right)$ as the set of all
the extremal symmetric states, that is the invariant states which are
ergodic w.r.t. the action of $\bp_J$.

In the section we investigate some of the basic ergodic properties
enjoyed by the symmetric states. To this aim, denote $M$ the
Cesaro Mean w.r.t. $\bp_J$, given for a generic object $f(g)$ by
\begin{equation*}
M\{f(g)\}:=\lim_{I\uparrow J}\frac1{|\bp_I|}\sum_{g\in \bp_I}f(g)\,,
\end{equation*}
provided the l.h.s. exists in the appropriate sense. As usual,
$I\subset J$ runs over all the finite parts of $J$.
\begin{thm}
\label{mainer}
Let $\f\in\cs_{\bp_J}(\carf(J))$. Then the
following assertions hold true.
\begin{itemize}
\item[(i)] The state $\f$ is even. \item[(ii)] The state $\f$ is
asymptotically Abelian in average:
$$
M\{\f(C[\a_g(A),B]D)\}=0\,,\quad A,B,C,D\in\carf(J)\,.
$$
\item[(iii)] $\f\in\ce\left(\cs_{\bp_J}(\carf(J))\right)$ if and only if it is weakly clustering:
$$
M\{\f(\a_g(A)B)\}=\f(A)\f(B)\,,\quad A,B\in\carf(J)\,.
$$
\end{itemize}
\end{thm}
\begin{proof}
(i) Let $A$ be localized and odd. Proposition \ref{fn} gives
\begin{align*}
&\{E_\f\pi_\f(A)E_\f,E_\f\pi_\f(A^*)E_\f\}\\
=M\{E_\f\pi_\f(A)U_\f(g)&\pi_\f(A^*)E_\f+E_\f\pi_\f(A^*)U_\f(g)\pi_\f(A)E_\f\}\\
=&M(E_\f\pi_\f(\{A,\a_g(A^*)\})E_\f)
\end{align*}
By Lemma \ref{stiro} and the CAR relations, one finds that the quantity
above is equal to zero, where the
limit in the Cesaro mean is understood in the strong operator
topology.\footnote{The analogous case based on the spatial translations has been
treated in \cite{BR2}, Example 5.2.21, where the particular form of the action of $\bz^d$
on $\carf(\bz^d)$ as the shift, allows
to reach directly the result without using Lemma \ref{stiro}.}
This implies that $E_\f\pi_\f(A)E_\f=0$, that is
$$
\f(A)=\langle E_\f\pi_\f(A)E_\f\Omega_\f,\Omega_\f\rangle=0\,.
$$
Hence $\f$ vanishes on the localized odd elements. Since
$\carf(J)\sim\carf_{+}(J)\bigoplus\carf_{-}(J)$ as a Banach space,
we can approximate a generic odd element $A$ with a sequence
$\{A_n\}_{n\in\bn}$ made by odd and localized elements. By the
above result, one obtains
$$
\f(A)=\f(\lim_nA_n)=\lim_n\f(A_n)=0\,.
$$
This means that $\f$ vanishes on all the odd elements, that is $\f$ is even.

(ii) The same computations as before show that, if $A$ or $B$ is even, then
\begin{equation}
\label{stirak}
[E_\f\pi_\f(A)E_\f,E_\f\pi_\f(B)E_\f]=0\,.
\end{equation}
By a standard approximation argument, we can reduce the matter to
localized elements. Fix $A$ even. Proposition \ref{fn}, Lemma
\ref{stiro} and \eqref{stirak}, give
\begin{align*}
&M\{\f(C\a_g(A)BD)\}=M\{\f(\a_g(A)CBD)\}\\
=&\langle\pi_\f(A)E_\f\pi_\f(CBD)\Om_\f,\Om_\f\rangle
=\langle\pi_\f(CBD)E_\f\pi_\f(A)\Om_\f,\Om_\f\rangle\\
=&M\{\f(CBD\a_g(A))\}=M\{\f(CB\a_g(A)D)\}\,.
\end{align*}
By considering $A$ odd and splitting $C$, $D$ in their even and
odd parts, the result is reached by similar computations as above.

(iii) The weak clustering condition is equivalent to
$\dim(\ch_\f^{\bp_J})=1$, and it is immediate to show that implies
ergodicity (see \cite{S}, Proposition 3.1.10). The converse
assertion follows from Theorem \ref{mainer1}, that is $\carf(J)$ is
$\bp_J$--abelian (see \cite{S}, Proposition 3.1.12).
\end{proof}
\begin{thm}
\label{mainer1} The group $\bp_J$ acts as a Large Groups of
Automorphisms on $\carf(J)$, and the $C^*$--dynamical system
$(\carf(J), \bp_J, \a)$ is $\bp_J$--abelian.
\end{thm}
\begin{proof}
By taking into account \eqref{stirak} and (i) of Theorem
\ref{mainer}, we conclude that $(\carf(J), \bp_J)$ is
$\bp_J$--Abelian. Fix an arbitrary integer $n\in\bn$,
$A,B_1,\dots,B_n,C\in\carf(J)$ and $\f\in\cs_{\bp_J}(\carf(J))$.
Now, by (ii) of Theorem \ref{mainer}, for each $\eps>0$ there
exists a finite set $I\subset J$ such that, for any $k=1,\dots,n$,
\begin{align*}
\left|\f\left(C^*\left[\left(\sum_{g\in \bp_I}\frac{\a_g(A)}{|\bp_I|}\right),B_k\right]C\right)\right|\\
=\left|\frac1{|\bp_I|}\sum_{g\in
\bp_I}\f(C^*[\a_g(A),B_k]C)\right|<\eps\,.
\end{align*}
This leads to the assertion by Theorem 3.5 of \cite{St1}.
\end{proof}
As $\bp_J$ is acting on $\carf(J)$ as a large group of
automorphisms, by Theorem 3.1 of \cite{St1}, for each state
$\f\in\cs_{\bp_J}(\carf(J))$ there exists a conditional
expectation
$$
\F_\f:\pi_\f(\carf(J))''\to\gb_{\bp_J}(\f)
$$
of the von Neumann algebra $\pi_\f(\carf(J))''$ onto
$\gb_{\bp_J}(\f)$. As the state $\f$ is asymptotically Abelian in
average (cf. (ii) of Theorem \ref{mainer}) we prove the following
result which, even if is not used in the sequel, may have an
interest in itself.
\begin{prop}
\label{avlgu}
Let $\f\in\cs_{\bp_J}(\carf(J))$ and $A\in\carf(J)$. Then
\begin{equation}
\label{cessnet}
\mathop{\rm w}-\lim_{I\uparrow J}
\frac1{|\bp_I|}\sum_{g\in
\bp_I}U_\f(g)\pi_\f(A)U_\f(g)^{-1}=\F_\f(\pi_\f(A))\,.
\end{equation}
\end{prop}
\begin{proof}
Let $\{I_\b\}\subset \{I\}$ any subnet of the Cauchy net $\{I\}$
of all the finite subsets $I\subset J$ such that
${\displaystyle\frac1{|\bp_{I_\b}|}\sum_{g\in
\bp_{I_\b}}U_\f(g)\pi_\f(A)U_\f(g)^{-1}}$ converges in the weak operator
topology, which exists by compactness. Fix an arbitrary
$h\in\bp_J$. Then there exists a $\b_h$ such that $I\supset
I_{\b_h}$ implies $h\in\bp_I$. We get
\begin{align*}
&\mathop{\rm w}-\lim_{I_\b} U_\f(h)\bigg(\frac1{|\bp_{I_\b}|}\sum_{g\in \bp_{I_\b}}U_\f(g)\pi_\f(A)U_\f(g)^{-1}\bigg)U_\f(h)^{-1}\\
=&\mathop{\rm w}-\lim_{I_\b\supset I_{\b_h}} \frac1{|\bp_{I_\b}|}\sum_{g\in \bp_{I_\b}}U_{\f}(hg)\pi_\f(A)U_{\f}(hg)^{-1}\\
=&\mathop{\rm w}-\lim_{I_\b\supset I_{\b_h}} \frac1{|\bp_{I_\b}|}\sum_{g\in \bp_{I_\b}}U_{\f}(g)\pi_\f(A)U_{\f}(g)^{-1}\\
=&\mathop{\rm w}-\lim_{I_\b} \frac1{|\bp_{I_\b}|}\sum_{g\in
\bp_{I_\b}}U_{\f}(g)\pi_\f(A)U_{\f}(g)^{-1}\,.
\end{align*}
Thus, each weak limit point of the Cesaro net on the l.h.s. of
\eqref{cessnet} is invariant. By using the asymptotic Abelianess
property (ii) of Theorem \ref{mainer}, and arguing as in the proof
of Lemma 5.3 of \cite{St1}, one shows that the limit above is in
$\gz_\f$. Then it belongs to $\gb_{\bp_J}(\f)$. But
$\gb_{\bp_J}(\f)$ can contain at most one of such limit points
since $\bp_J$ acts on $\carf(J)$ as a large group of automorphisms
(cf. proof of Theorem 3.1 in \cite{St1}). As a consequence, the
limit is precisely $\F_\f(\pi_\f(A))$.
\end{proof}

\section{the structure of the symmetric states: de finetti theorem}
\label{ss5}

The present and the following sections are mainly concerned with
the characterization of the extremal symmetric states. In
particular here we consider the case in which $J$ is countable,
hence $J\equiv\bn$. We preliminary report the definition of the
sequence of permutations $\{g_n\}_{n\in\bn}$ in \cite{St2} given
by
\begin{equation}
\label{mixi}
g_{n}\left( k\right) :=\left\{
\begin{array}{ll}
2^{n-1}+k & \text{if }1\leq k\leq 2^{n-1}\,, \\
k-2^{n-1} & \text{if }2^{n-1}<k\leq 2^{n}\,,\\
k & \text{if }2^{n}<k\,.
\end{array}
\right.
\end{equation}
\begin{defin}
\label{brid}
A state $\f\in\cs(\carf(\bn))$ is said to be {\it strongly
clustering} if, for every A, B $\in\carf(\bn)$
$$
\lim_n\f(\a_{g_n}(A)B)=\f(A)\f(B)\,.
$$
\end{defin}
\begin{lem}
\label{scaz1}
If $\f\in\cs_{\bp_{\bn}}(\carf(\bn))$ is extremal, then for each $A\in\carf(\bn)$
$$
\mathop{\rm w}-\lim_{n}[\pi_\f(\a_{g_n}(A))\Om_\f]=\f(A)\Om_\f\,.
$$
\end{lem}
\begin{proof}
By a standard approximation argument, we can reduce the matter to
$A\in\carf_0(\bn)$.  Let $g\in\bp_\bn$. As shown in the proof of
Lemma 2.6 of \cite{St2}, there exists $n_{A,g}$ such that
$n>n_{A,g}$ implies $\a_{gg_n}(A)=\a_{g_n}(A)$. This means that
any weak limit point (which exists by compactness) of the sequence
$\big\{\pi_\f(\a_{g_n}(A))\Om_\f\big\}$ is an invariant vector
under the action of $\bp_\bn$, that is it belongs to
$\ch_{\f}^{\bp_\bn}$. Let
$$
\xi:={\rm
w}-\lim_{k}U_\f(g_{n_{\xi}(k)})\pi_\f(A)U_\f(g_{n_{\xi}(k)})^{-1}\Om_\f\equiv{\rm
w}-\lim_{k}\pi_\f(\a_{g_{n_{\xi}(k)}}(A))\Om_\f
$$
be one of such limit points. Since $\xi$ is a vector in
$\ch_{\f}^{\bp_\bn}$ and $\f$ is extremal, by (iii) of Theorem
\ref{mainer}, one has $\xi=\G(A,\xi)\Om_\f$. By using
(\ref{mixi}), we obtain
\begin{align*}
\G(A,\xi)=&\big\langle\lim_{k}U_\f(g_{n_{\xi}(k)})\pi_\f(A)U_\f(g_{n_{\xi}(k)})^{-1}\Om_\f,\Om_\f\big\rangle\\
=&\lim_{k}\big\langle\pi_\f(\a_{g_{n_{\xi}(k)}}(A)\Om_\f,\Om_\f\big\rangle=\f(A)\,.
\end{align*}
Namely, there is only one of such weak limit points in $\ch_\f$,
which is $\f(A)\Om_\f$.
\end{proof}
\begin{thm}
\label{defin}
Let $\f\in\cs_{\bp_{\bn}}(\carf(\bn))$. Then the
following are equivalent.
\begin{itemize}
\item[(i)] $\f\in\ce(\cs_{\bp_{\bn}}(\carf(\bn)))$,
\item[(ii)] $\f$ is strongly clustering,
\item[(iii)] ${\displaystyle\f=\prod_\bn\r}$ for some even state $\r\in\bm_2(\bc)$.
\end{itemize}
\end{thm}
\begin{proof}
$(i)\Rightarrow(ii)$ Suppose $\f$ is extremal and take
$A,B\in\carf(\bn)$. Then by Lemma \ref{scaz1}, we get
$$
\lim_n\f(A\a_{g_n}(B))=\lim_n\langle\pi_\f(\a_{g_n}(B))\Om_\f,\pi_\f(A^*)\Om_\f\rangle=\f(A)\f(B)\,,
$$
that is $\f$ is strongly clustering.

$(ii)\Rightarrow(i)$  Choose a vector $\xi\perp\Om_\f$ belonging
to $\ch_\f^{\bp_{\bn}}$, and fix $\eps>0$. Then there exists
$B\in\carf(\bn)$ such that $\|\xi-\pi_\f(B)\Om_\f\|<\eps/2$. Let
$A\in\carf(\bn)$ such that $\|\p_\f(A)\|\leq1$. We get
\begin{align*}
&|\langle\xi,\pi_\f(A)\Om_{\f}\rangle|=|\langle U_{\f}(g_n)\pi_\f(A^*)U_{\f}(g_n)^{-1}\xi,\Om_{\f}\rangle|\\
\leq&|\langle U_{\f}(g_n)\pi_\f(A^*)U_{\f}(g_n)^{-1}\pi_\f(B)\Om_\f,\Om_\f\rangle|+\eps/2
=|\f(\a_{g_n}(A^*)B)|+\eps/2\,.
\end{align*}
Suppose now $\f$ is strongly clustering. By taking the limit for
$n\rightarrow\infty$ on both sides, we get
$$
|\langle\xi,\pi_\f(A)\Om_{\f}\rangle|\leq|\f(A^*)||\f(B)|+\eps/2
<\eps\,.
$$
As $\eps>0$ is arbitrary and $\Om_\f$ is cyclic for
$\pi_\f(\carf(\bn))$, we get $\xi=0$. This means that
$\ch_\f^{\bp_{\bn}}$ is one dimensional, which implies (and it is
indeed equivalent by Proposition 3.1.12 of \cite{S}) that $\f$ is
extremal.

Obviously, if $\f$ is a product state of a single state as in
$(iii)$, then it is strongly clustering. Suppose now that $(ii)$
holds true. After fixing $j\in\bn$, we start by identifying
$\carf(\{j\})$ with $\bm_2(\bc)$ and, as usual, denoting
$\io_j:\bm_2(\bc)\to\carf(\bn)$ the related embedding.

By Theorem 1 of \cite{AM2}, any product state is uniquely
determined by the product of the values of the state on the
generators. Hence, it is enough to check, for each $n\in\bn$ and
$A_1,\dots,A_n\in\bm_2(\bc)$,
\begin{equation}
\label{prod1bis}
\f(\io_1(A_1)\cdots \io_n(A_n))=\prod_{j=1}^n\r(A_j)\,,
\end{equation}
The proof now proceeds as in Theorem 2.7 of \cite{St2}. We give
the details with the appropriate modifications. Define for
$j\in\bn$, $\rho_j(A):=\f(\io_j(A))$. As $\f$ is even (cf. Theorem
\ref{mainer}.(i)), all the $\rho_j$ are even. In addition, for
$i,j\in\bn$, $\rho_i=\rho_j=:\rho$ as $\f$ is symmetric. Equation
\eqref{prod1bis} can be achieved by an induction procedure. Indeed,
for $n=1$ it follows immediately, so we suppose it holds true till
$n-1$. Fix $\eps>0$. By using the inductive hypothesis and the
strong clustering property , we get
\begin{align}
\label{cinque}
&\bigg|\f(\io_1(A_1)\cdots \io_n(A_{n-1})\a_{g_m}(\io_n(A_n)))-\prod_{j=1}^n\r(A_j)\bigg|\\
=|\f(\io_1(A_1)&\cdots\io_n(A_{n-1})\a_{g_m}(\io_n(A_n)))-\f(\io_1(A_1)\cdots
\io_n(A_{n-1}))\r(A_n)| <\eps\nn
\end{align}
for some $m>n$. Choose now a permutation $g\in\bp_\bn$ such that
$g(j)=j$ if $1\leq j<n$ and $g(n) =g_{m}(n) $. Then
\begin{align*}
\f(\io_1(A_1)\cdots \io_n(A_{n}))=&\f(\a_{g}(\io_1(A_1)\cdots \io_n(A_{n})))\\
=&\f(\io_1(A_1)\cdots \io_n(A_{n-1})\a_{g_m}(\io_n(A_n)))
\end{align*}
which, combined with \eqref{cinque}, leads to the assertion as $\eps>0$ is arbitrary.
\end{proof}
When a state is extremal among the symmetric ones, the result of
Proposition \ref{avlgu} can be strengthened as we see in the
following proposition, inserted for the sake of completeness.
\begin{prop}
\label{scaz}
If $\f\in\ce(\cs_{\bp_{\bn}}(\carf(\bn)))$. Then for each $A\in\carf(\bn)$,
\begin{equation}
\label{aavvrr}
\mathop{\rm w}-\lim_{n}
U_\f(g_n)\pi_\f(A)U_\f(g_n)^{-1}=\f(A)\idd=\F_\f(\pi_\f(A))\,.
\end{equation}
\end{prop}
\begin{proof}
By a standard approximation argument, it is enough to consider
$\xi=\pi_\f(B)\Om_\f$, $\eta=\pi_\f(C^*)\Om_\f$, and reduce the
matter to $A,B,C\in\carf_0(\bn)$. Let $A=A_++A_-$, $B=B_++B_-$ the
split of $A$, $B$ into the even and odd part. By Theorem \ref{defin}, it is enough to assume that $\f$ is
strongly clustering. As $\f$ is even, by using the
standard (anti)commutation relations, we get
\begin{align*}
&\lim_n\langle U_\f(g_n)\pi_\f(A)U_\f(g_n)^{-1}\xi,\eta\rangle=\lim_n\f(C\a_{g_n}(A)B)\\
=&\lim_n\f(CB\a_{g_n}(A_+))+\lim_n\f(CB_+\a_{g_n}(A_-))-\lim_n\f(CB_-\a_{g_n}(A_-))\\
=&\f(CB)\f(A_+)+\f(CB_+)\f(A_-)-\f(CB_-)\f(A_-)\\
=&\f(CB)\f(A_+)+\f(CB_+)\f(A_-)\\
=&\f(CB)\f(A_+)+\f(CB_+)\f(A_-)+\f(CB_-)\f(A_-)\\
=&\f(CB)\f(A)=\langle(\f(A)\idd)\xi,\eta\rangle\,.
\end{align*}
This means that the above weak limit is in $\gb_{\bp_\bn}(\f)$.
But, as observed in the proof of \eqref{cessnet}, it is nothing
else than $\F_\f(\pi_\f(A))$.
\end{proof}
Notice that, in the case of the $C^*$--tensor product (i.e. when
the system is asymptotically Abelian in norm) \eqref{aavvrr} is
satisfied for each symmetric state, see \cite{St2}, Lemma 2.6. In
our situation, we have merely the expected average property \eqref{cessnet}, and \eqref{aavvrr}
is satisfied only for extremal symmetric states.

Having characterized the extremal symmetric states as the
Araki--Moriya product states, we are interested in the ergodic
decomposition of an invariant state under the action of the
permutation group. Namely, we want to express every symmetric state
as the barycenter of a unique (maximal) Radon probability measure
on $\cs_{\bp_{\bn}}(\carf(\bn))$ whose support is
$\ce(\cs_{\bp_{\bn}}(\carf(\bn)))$. The existence of such a
measure is granted, since $\cs_{\bp_{\bn}}(\carf(\bn))$ is
metrizable (see \cite{BR1}, Proposition 4.1.3 and Theorem 4.1.11).
It is unique if and only if $\cs_{\bp_{\bn}}(\carf(\bn))$ is a
Choquet simplex (see \cite{BR1}, Theorem 4.1.15). This is our
case, as we see in the following result.
\begin{thm}
\label{csyx}
$\cs_{\bp_{\bn}}(\carf(\bn))$ is a Choquet simplex. Then, for each
$\f\in\cs_{\bp_{\bn}}(\carf(\bn))$ there exists a unique maximal
Radon probability measure $\m$ on $\cs_{\bp_{\bn}}(\carf(\bn))$
such that
$$
\f(A)=\int\psi(A)d\m(\psi)\,,\quad A\in\carf(\bn)\,,
$$
and
$\m(\ce(\cs_{\bp_{\bn}}(\carf(\bn)))=1$.
\end{thm}
\begin{proof}
By Theorem \ref{mainer1}, $(\carf(\bn), \bp_\bn)$ is $\bp_\bn$--Abelian,
then, as a consequence of Theorem 3.1.14 of \cite{S},
$\cs_{\bp_{\bn}}(\carf(\bn))$ is a Choquet simplex. The last part
is a rephrasing of the above discussion.
\end{proof}
The property for the set of extremal states to be $*$--weakly
closed implies a nice result. In fact, in \cite{B} it is shown
that a simplex with closed boundary is affinely isomorphic to the
probability measures on a compact Hausdorff space. These facts are
stated in the following proposition, whose proof, based on Lemma
2.2, Theorem 5.3 and arguments analogous to those developed in
Theorem 2.8 of \cite{St2}, is left to the reader.
\begin{prop}
\label{boundary} The Choquet simplex $\cs_{\bp_{\bn}}(\carf(\bn))$
has a $*$--weakly closed boundary and is affinely isomorphic to
the probability measures on a closed interval.
\end{prop}
Put $\ga:=\carf(\bn)$. Let $\f\in\cs(\ga)$. Obviously,
$\pi_\f(\ga)''$ is a hyperfinite von Neumann algebra. A state $\f$
on a $C^{*}$-algebra $\ga$ is called {\it factor state} if the
double commutant $\pi_\f(\ga)^{''}$ of $\pi_\f(\ga)$ is a factor.
The state $\f$ is said of type $X$ if $\pi_\f(\ga)''$ is of type
$X$, where $X=\ty{I}_{\infty}$, $\ty{II_{1}}$, $\ty{II_{\infty}}$,
$\ty{III}$ or $\ty{III_{\l}}$, $\l\in[0,1]$ according to the
Connes' classification of the type $\ty{III}$ factors \cite{C}.
\begin{prop}
\label{factor}
Let $\r_\m$ be a even state in $\bm_2(\bc)$ with
eigenvalues $\m$ and $1-\m$, and
${\displaystyle\f_\m:=\prod_\bn\r_\m}$ the corresponding product state on $\carf(\bn)$. Then
\begin{itemize}
\item[(1)] $\f_\m$ is a factor state of type $\ty{I_{\infty}}$
if and only if $\m=0$ or $\m=1$.
\item[(2)]  $\f_\m$ is a factor state of type $\ty{II_{1}}$  if and only
if $\m=1/2$.
\item[(3)]  $\pi_{\f_\m}(\ga)''$ is of type $\ty{III_{\l}}$ if and only
if $0<\m<1/2$ and $\l=\frac{\mu }{1-\mu }$, or $1/2<\m<1$ and
$\l=\frac{1-\mu }{\mu }$.
\end{itemize}
\end{prop}
\begin{proof}
By Lemma \ref{matr2}, we get (under the same notations)
$\pi_{\f_\m}(\ga)''\sim\pi_{\om_\m}(\gb)''$. But the
$\pi_{\om_\m}(\gb)''$ are factors whose type is determined by the
ratio between the smallest and the largest eigenvalue of the trace
operator describing $\r_\m$, according to the three possibilities
listed above in the statement. The reader is referred to A.17 of
\cite{St} and the references cited therein.
\end{proof}
Notice that we do not have the type $\ty{II_{\infty}}$, and
furthermore, the type $\ty{II_{1}}$ occurs only for $\f_{1/2}$. Thus, the latter gives rise the trivial face made of a
singleton.\footnote{The $\ty{II_{\infty}}$ might appear for the more general situation
investigated in \cite{F1}, where the regrouped CAR algebra with $2d$ generators is attacked to each site.}
The cases of the portions corresponding to $\ty{I_{\infty}}$ and
$\ty{III}$ are covered by the next result. Recall that a face of a
given simplex $K$ is a convex subset $F$ of $K$ such that, if
$\chi\in F$, $\psi\in K$ and for $\m>0$, $\psi\leq\m\chi$, implies
$\psi\in F$. The proof of the forthcoming proposition follows {\it
mutatis mutandis} the lines of the analogous Lemma 2.9 in
\cite{St2}. The details are left to the reader.
\begin{prop}
\label{fdczzo}
If $X$ denotes $\ty{I_{\infty}}$ or $\ty{III}$, and
$$
\cs_{\bp_{\bn}}(\carf(\bn))_X:=\{\f\in
\cs_{\bp_{\bn}}(\carf(\bn))\mid\,\f\,\,{\rm is\,of\,type}\,\,X\}\,,
$$
then
$\cs_{\bp_{\bn}}(\carf(\bn))_X$ is a face of
$\cs_{\bp_{\bn}}(\carf(\bn))$.
\end{prop}
Recalling that the boundary of an arbitrary nontrivial compact
convex $K$ is nonvoid by the Krein--Milman Theorem, by
Propositions \ref{boundary} and \ref{factor}, one has that
$\ce(\cs_{\bp_{\bn}}(\carf(\bn))_{\ty{III}})$ consists of two open
connected components of $\ce(\cs_{\bp_{\bn}}(\carf(\bn)))$,
$\ce(\cs_{\bp_{\bn}}(\carf(\bn))_{\ty{I_{\infty}}})$ is given by
two states, that is the pure states, whereas $\ce(
\cs_{\bp_{\bn}}(\carf(\bn))_{\ty{II_{1}}})$ is made by one state,
that is the unique trace.

We end the section with a brief sketch on the extension of the
situation in Proposition \ref{fdczzo} to more general cases and
leave further details to the reader.

Let $\ga$ be a separable $C^*$--algebra and $E\in(0,1)$ be a Borel set. Fix a state $\f\in\cs(\ga)$ and put
$M:=\pi_\f(\ga)''$. Let
\begin{equation}
\label{bbbb}
M=\int^\otimes_\G M_\g\di\n(g)
\end{equation}
be its direct integral decomposition into von Neumann factors. Here, $(\G,\n)$
is a standard probability measure space which can be chosen as
$(\spec(\gz_\f),\n_{\om\lceil_{\gz_\f}})$, $\spec(\gz_\f)$ and
$\om$ being the spectrum of $\gz_\f$, and a faithful normal state
on $M$ respectively, see e.g. \cite{S, T1}.

We say that the state $\f$ is of type $X_E$ if $M$ contains only
type $\ty{III_{\l}}$ factors for some $\l\in E$ in its direct
integral decomposition \eqref{bbbb}. Let $\G_E\subset\G$ be the
subset made of the $\g$ such that $M_\g$ is a type $\ty{III_{\l}}$
factor for some $\l\in E$. By Theorem 2.2 of \cite{Su} (see also
(ii) in Theorem 21.2 of \cite{N}), $\G_E$ is a $\n$--measurable
set. The fact that $\f$ is of type $X_E$ simply means that
$\n(\G_E)=1$. Put
$$
\cs_{\bp_{\bn}}(\ga)_{X_E}:=\{\f\in
\cs_{\bp_{\bn}}(\ga)\mid\,\f\,\,{\rm is\,of\,type}\,\,X_E\}\,.
$$
\begin{rem}
\label{remm5}
By using the same lines as in Lemma 2.9 of \cite{St2}, we can show that
$\cs_{\bp_{\bn}}(\carf(\bn))_{X_E}$ is a face of
$\cs_{\bp_{\bn}}(\carf(\bn))$.
\end{rem}
In fact, let $\f=\l\f_1+(1-\l)\f_2$, $\l\in[0,1]$ be a convex
combination of states in $\cs_{\bp_{\bn}}(\carf(\bn))_{X_E}$. As
in the proof of Lemma 2.9 of \cite{St2}, we find
$P_1,P_2\in\gz_\f$ such that
$$
\pi_{\f_k}(\ga)''=P_k\pi_{\f}(\ga)''\,,\quad k=1,2\,.
$$
Of course, $P_1\pi_{\f}(\ga)''$ contains only type $\ty{III_{\l}}$
factors, $\l\in E$, in its direct integral factor decomposition
\eqref{bbbb}. The same happens to $(P_2-P_1P_2)\pi_{\f}(\ga)''$ as
$P_2-P_1P_2\leq P_2$. Since one can show that $P_1+P_2-P_1P_2=I$,
it follows that $\f$ is of type $X_E$. So
$\cs_{\bp_{\bn}}(\carf(\bn))_{X_E}$ is closed under convex
combinations. Let now take $\om\in\cs_{\bp_{\bn}}(\carf(\bn))$,
$\f\in\cs_{\bp_{\bn}}(\carf(\bn))_{X_E}$ with $\om\leq\l\f$ for
some $\l>0$. As before, there exists a projection $P\in\gz_\f$
such that $\pi_{\om}(\ga)''=P\pi_{\f}(\ga)''$, which implies that
$\om\in\cs_{\bp_{\bn}}(\carf(\bn))_{X_E}$.

\section{symmetric states in the uncountable case}
\label{sec6}

The starting point to obtain the characterization of the extremal
symmetric states is the clustering property given in Definition
\ref{brid}. It is the bridge between the extremality and the
property to being a product state (cf. Theorem \ref{defin}). Suppose
that $\carf(J)$ is generated by infinitely uncountably many
annihilators, that is $J$ is uncountable. Fix a state
$\f\in\cs(\carf(J))$. We generalize the clustering property in
the following way.

We start with a pair $\ci:=(I,\n)$, where $I\subset J$ is
countable and $\n:I\to\bn$ is a bijection defining an order
$i_1,i_2,\dots$ on $I$. Define the sequence
$\{g^\ci_n\}_{n\in\bn}\subset\bp_J$ as $g^\ci_n:=g_{\n(i_n)}$,
where $g_k\in\bp_\bn$ is given in \eqref{mixi}. Fix now
$\f\in\cs(\carf(J))$. It is said {\it strongly clustering} if,
for each $A,B\in\carf(J)$ there exists $\ci:=(I,\n)$ with
$A,B\in\carf(I)$ such that
\begin{equation}
\label{wlunct1} \lim_n\f(\a_{g^\ci_n}(A)B)=\f(A)\f(B)\,.
\end{equation}

As a preliminary fact we note that, if $\ci:=(I,\n)$ with
$A\in\carf(I)$ localized, and $g\in\bp_J$, then there exists
$n_{A,g}$ such that $n>n_{A,g}$ implies
$\a_{gg^\ci_n}(A)=\a_{g^\ci_n}(A)$. In fact, as $g$ changes only a
finite number of indices in $I$, say up to $m$, and $A$ is
localized, say in the first $s$ elements of $I$, it is enough to
choose $n$ such that $\max\{m,s\}\leq2^{n-1}$. By using this fact,
Lemma \ref{scaz1} holds true also in the present situation.
Namely, fix $A\in\carf(J)$. Then for each $\ci=(I,\n)$ such that
$A\in\carf(I)$
\begin{equation}
\label{wlunct}
\mathop{\rm
w}-\lim_{n}[\pi_\f(\a_{g^\ci_n}(A))\Om_\f]=\f(A)\Om_\f\,,
\end{equation}
provided $\f$ is extremal.

Now we are ready to extend Theorem \ref{defin} to uncountable case.
\begin{thm}
\label{defin1}
Let $\f\in\cs_{\bp_{J}}(\carf(J))$. Then $\f\in\ce(\cs_{\bp_{J}}(\carf(J)))$ if and only if
${\displaystyle\f=\prod_J\r}$ for some even state $\r\in\bm_2(\bc)$.
\end{thm}
\begin{proof}
If $\f$ is a product state, it is obviously strongly clustering.
Suppose now $\f$ be strongly clustering and fix a sequence
$\{A_i\}_{i\in\bn}\subset\bm_2(\bc)$. The proof proceeds by
induction. Choose $\ci:=(I,\n)$ such that
$A:=\iota_{j_1}(A_1)\cdots\iota_{j_{n-1}}(A_{n-1})$,
$B:=\iota_{j_n}(A_n)$ belong to $I$ and \eqref{wlunct1} holds
true. After noticing that
$$
\f(g^\ci_n(A)B)=\f\lceil_{\carf(I)}(g^\ci_n(A)B)\,.
$$
we can reduce the matter to the countable situation, that is to
Theorem \ref{defin}, as $\carf(I)\sim\carf(\bn)$. Thus a symmetric
state is a product one if and only if it is strongly clustering.
On the other hand, if $\f$ is extremal we conclude by
\eqref{wlunct} that $\f$ is strongly clustering. For the reverse
implication we reason as in Theorem \ref{defin} as well. Indeed,
choose a vector $\xi\perp\Om_\f$ belonging to $\ch_\f^{\bp_J}$ and
fix $\eps>0$. Then there exists $B\in\carf(J)$ such that
$\|\xi-\pi_\f(B)\Om\|<\eps/2$. Let $A\in\carf(J)$ such that
$\|\p_\f(A)\|\leq1$. We get for each $\ci=(I,\n)$ such that
$A,B\in\carf(I)$,
\begin{equation}
\label{wlunct2}
|\langle\xi,\pi_\f(A)\Om_{\f}\rangle|\leq|\f(\a_{g^{\ci}_n}(A^*)B)|+\eps/2\,.
\end{equation}
As $\f$ is strongly clustering, there exists an $\ci:=(I,\n)$ such
that $A,B\in\carf(I)$ and \eqref{wlunct1} holds. Taking the limit
on both sides in \eqref{wlunct2}, one obtains
$$
|\langle\xi,\pi_\f(A)\Om_{\f}\rangle|\leq\lim_n|\f(\a_{g^\ci_n}(A^*)B)|+\eps/2
=|\f(A^*)||\f(B)|+\eps/2<\eps\,.
$$
We conclude as before that $\xi=0$ and then $\f$ is extremal.
Since we verified that both properties (i.e. being a product state
and extremeness) are equivalent to strong clustering, the proof is
complete.
\end{proof}
\begin{rem}
Theorem \ref{defin1} holds true {\it mutatis mutandis} for the case of infinite tensor
product of $C^*$--algebras considered in \cite{St2}, when the index set is uncountable.
\end{rem}
In the general (uncountable) case it is possible to achieve the
ergodic decomposition of a symmetric state $\f$ as in Theorem 5.5.
(countable case). In fact, the convex of the symmetric states on
$\carf(J)$ is yet a Choquet simplex (see \cite{S}, Theorem
3.1.14). The difference w.r.t. the countable case consists in the
fact that $\cs_{\bp_J}(\carf(J))$ is not metrizable. Then the
unique maximal decomposing measure $\m$ of $\f$ is only
pseudosupported on $\ce(\cs_{\bp_J}(\carf(J)))$. Namely, we still
have
$$
\f(A)=\int\psi(A)d\m(\psi)\,,\quad A\in\carf(J)\,,
$$
but here $\ce(\cs_{\bp_{J}}(\carf(J))$ is merely a Borel set which
is not a Baire one. Then the maximal measure $\m$  satisfies
$\m(B)=1$ for each Baire set $B$ containing
$\ce(\cs_{\bp_{J}}(\carf(J))$ (see \cite{BR1}, Theorem 4.1.11).

Finally, one realizes that most of the results at the end of
Section \ref{ss5}, except Remark \ref{remm5}, may be extended to
the general (uncountable) case. For example, as in Proposition
\ref{boundary}, one sees that the simplex of the symmetric states
on $\carf(J)$ has a $*$--weakly closed boundary and is affinely
isomorphic to the regular probability measures on the unit
interval.
\section*{Acknowledgements}
The second--named author would like to thank Simion Stoilow Institute of Mathematics
of the Romanian Academy for the warm hospitality, thanks to a BITDEFENDER invited professorship
position, where part of the present work has been done.

\end{document}